\documentclass[12pt]{article}
\usepackage{amssymb}
\usepackage{graphicx}
\usepackage{amsmath}
\usepackage{harvard}

\newtheorem{theorem}{Theorem}

\newtheorem{condition}[theorem]{Condition}

\newtheorem{corollary}[theorem]{Corollary}

\newtheorem{example}[theorem]{Example}

\newtheorem{lemma}{Lemma}

\newtheorem{remark}[theorem]{Remark}

\newenvironment{proof}[1][Proof]{\textbf{#1.} }{\ \rule{0.5em}{0.5em}}
\input{tcilatex}
\renewcommand{\baselinestretch}{1}
\begin{document}

\title{Uniqueness of Solutions for Certain Markovian Backward Stochastic
Differential Equations}
\author{Coskun Cetin\thanks{%
CSU, Department of Mathematics and Statistics, 6000 J St., Sacramento, CA
95819. Ph: (916) 278-6221. Fax: (916) 278-5586. Email: cetin@csus.edu}}
\date{October 28, 2012}
\maketitle

\begin{abstract}
This paper considers the problem of uniqueness of the solutions to a class
of Markovian backward stochastic differential equations (BSDEs) which are
also connected to certain nonlinear partial differential equation (PDE)
through a probabilistic representation. Assuming that there is a solution to
the BSDE or to the corresponding PDE, we use the probabilistic
interpretation to show the uniqueness of the solutions, and provide an
exampleof a stochastic control application.
\end{abstract}

\begin{quotation}
 \medskip

 \smallskip 
\end{quotation}

\section{1. Introduction}

In this paper, we study a class of decoupled forward-backward stochastic
differential equations (FBSDEs) which have a Markovian structure of the
following form:%
\begin{eqnarray}
dX(t) &=&\mu (t,X(t))dt+\sigma (t,X(t))dW(t),\text{ }0\leq t\leq T  \notag \\
dY(t) &=&-F(t,X(t),Y(t),Z(t))dt+Z(t)dW(t),\text{ }0\leq t\leq T  \label{p0}
\\
X(0) &=&x_{0};\text{ \ }Y(T)=g(X(T))  \notag
\end{eqnarray}%
where the forward process $X$ has a unique solution in a probability space $%
(\Omega ,\digamma ,P),$ the random variable $Y(T)=g(X(T))$\ is integrable
and the \textit{driver} of the backward process $Y,$ $F(t,x,y,z),$\ is
quadratic in $z$. Due to such a growth condition on $z$, these BSDEs are
called \textit{quadratic} BSDEs or "BSDEs with quadratic growth" in the
literature.\ Moreover, due to the Markovian nature of formulation, the
FBSDEs of the form (\ref{p0}) are known to be related to certain quasilinear
parabolic partial differential equations (PDEs).

After the first existence-uniqueness result for nonlinear BSDEs with
Lipshitz coefficients was given by Pardoux and Peng (1990), FBSDEs and
especially Markovian BSDEs have appeared in many application areas including
mathematical finance, stochastic optimal control and analysis of nonlinear
PDEs. The existence-uniqueness results for more general BSDEs were provided
by Mao (1995), Lepeltier and San Martin (1997, 1998), Kobylanski (2000),
Briand et. al (2007), Briand and Hu (2006, 2008) and Fan and Jiang (2010),
among others.\ Their connections with quasilinear PDE's were first stated by
Pardoux and Peng (1992), and Peng (1992) by generalising the Feynman-Kac
representation of PDE's. They also provided a uniqueness result when the
coefficients involved were uniformly Lipshitz. Similar results and their
connections with the stochastic control problems were also reported in El
Karoui et al (1997), Ma and Yong (1999), Cetin (2005), Fuhrman et. al (2006)
and Richou (2011).

The existence results for the quadratic BSDEs usually assume strong growth,
monotonicity, convexity/concavity or boundedness conditions on the driver or
on the terminal value. The issue of uniqueness is much more complicated and
usually requires stronger assumptions or some specific forms of the
parameters. See Briand et. al (2007), Fan and Jiang (2010) and Richou (2011)
for a discussion of such special cases, and the other works in the
literature. Our aim is to obtain the uniqueness results for a class of the
Markovian BSDEs with quadratic growth where the solution $Y$ is bounded from
below only. Such equations usually appear in the stochastic control
problems, where the process $Y$ would yield the value function of a
minimization problem over a suitable space of admissible controls. An
application to perturbed linear-quadratic regulator (LQR) problem is
provided in the last section.

The rest of the paper is organized as follows: The basic definitions and the
notations of the paper are introduced in the subsection 1.1 below. A
uniqueness result for solutions to a class of Markovian BSDEs is given in
the section 2. The section 3 describes how such BSDEs can be used to study
the properties of the solutions to some certain quasilinear PDEs which are
also related to the stochastic optimal control problems where only the drift
term of the state process is control-dependent. 

\subsection{1.1\quad Definitions and Notations}

For simplicity, we consider the one-dimensional Euclidean space $%
\mathbb{R}
$\ even though most of the results hold for higher dimensions. For a given $%
T>0$ and a probability space $(\Omega ,\digamma ,P)$ where $\digamma
=\{\digamma _{t}:0\leq t\leq T\}$ is the complete $\sigma -$algebra
generated by a standard Brownian motion process $W$, we define the following
spaces:

\begin{itemize}
\item $C^{p,q}([0,T])$: The space of all real-valued measurable functions $f$
$:[0,T]\times 
\mathbb{R}
$ such that $f(t,x)$ is $p$ (respectively, $q$) times continuously
differentiable with respect to $t$ (respectively, $x$) where $p,q$ are
non-negative integers.

\item $L_{\digamma _{T}}^{p}(\Omega )$: The space of $\digamma _{T}$%
-measurable random variables $H$\ such that $E[\left\vert H\right\vert
^{p}]<\infty $.

\item $L_{\digamma _{T}}^{\infty }(\Omega )$: The space of $\digamma _{T}$%
-measurable essentially bounded\ random variables.

\item $L_{\digamma }^{p}([0,T])$: The space of $\digamma $-adapted processes 
$f$ such that $E[\int\limits_{0}^{T}\left\vert f(t)\right\vert
^{p}dt]<\infty $.

\item $L_{\digamma }^{\infty }([0,T])$: The space of $\digamma $-adapted
essentially bounded processes.

\item $S_{\digamma }^{p}([0,T])$: The space of $\digamma $-adapted processes
such that $E[\sup\limits_{0\leq t\leq T}\left\vert f(t)\right\vert
^{p}]<\infty .$
\end{itemize}

The notation $E_{t}[.]$ will denote the conditional expectation $%
E[.|\digamma _{t}]$. When the initial value of a process $X$ is given at
time $t$, then $E^{t,x}[.]\ $refers to $E[.]\ $with $X_{t}=x$. For a
deterministic function $h(t,x):[0,T]\times 
\mathbb{R}
\rightarrow 
\mathbb{R}
$, the subscript notation denotes partial derivatives: $h_{t}(t,x)=\frac{%
\partial h}{\partial t}(t,x)$, $h_{x}(t,x)=\frac{\partial h}{\partial x}(t,x)
$ and $h_{xx}(t,x)=\frac{\partial ^{2}h}{\partial x^{2}}(t,x)$. In
particular, for functions or ODE's of one variable $t$, dot\ ($^{\cdot }$)
designates the derivative with respect to $t$. For a function $v\in
C^{1,2}([0,T]\times 
\mathbb{R}
)$, let $L$\ denote the backward evolution operator associated with the
forward diffusion process $X$ in (\ref{p0}): 
\begin{equation}
\mathsf{L}v(s,x)=v_{s}(s,x)+\mu (s,x)v_{x}(s,x)+\frac{1}{2}\sigma
^{2}v_{xx}(s,x).  \label{p1}
\end{equation}%
Then consider the PDE%
\begin{eqnarray}
\mathsf{L}v(t,x)+F(t,x,v,\sigma v_{{\large x}}) &=&0  \label{p2} \\
v(T,x) &=&g(x).  \notag
\end{eqnarray}%
If $\exists $ $c>0$ such that $\sigma (t,x)\geq c$ for all $(t,x)\in \lbrack
0,T]\times 
\mathbb{R}
,$ then the PDE (\ref{p2}) is called \textit{uniformly parabolic}. Such PDEs
are known to have unique classical or generalized (e.g. viscocity) solutions
under certain regularity and growth conditions. When a PDE is associated
with a stochastic control problem in the form of Hamilton-Jacobi-Bellman
(HJB in short) PDE, a "guess" solution to the HJB PDE usually turns out to
be the solution to the corresponding control problem, thanks to the
availability of a relevant verification theorem. For a summary of known
results and the assumptions on such verification theorems, see Fleming and
Soner (2006, IV.4) or Yong and Zhou (1999). A verification theorem is often
stated heuristically in applications to conclude that the solution to the
control problem is also the unique solution to the corresponding PDE, in a
suitable space of continuous functions. In this paper, our emphasis is on a
probabilistic description and interpretation of such equations.\emph{\ }

\section{2. A Uniqueness Result for a Class of Markovian BSDEs}

In this section, we first assume that the BSDE in (\ref{p0})\ has a solution 
$(Y,Z)$ in $S_{F_{T}}^{1}\times L_{F}^{2}$ in a probability space $(\Omega
,\digamma ,P)$. Even though an interpretation of the weak solutions of the
state variable $X$ is relevant in the PDE formulation, we are going to stick
to the strong existence-uniqueness in the reference space $(\Omega ,\digamma
,P)$, for the simpliciy of the presentation. The following result which is a
special case of the Bihari's inequality will be useful in the specification
of the assumptions and the proof of our main result. For a more general
version, one can refer to Bihari (1956) or Mao (1995).

\begin{lemma}[Bihari's inequality]
\label{Bihari}For $T>0$, let $f(t)$ and $v(t)$ be two continuous functions
on $[0,T]$. Moreover, let $\kappa :[0,\infty )\rightarrow \lbrack 0,\infty )$
be a continuous and nondecreasing function such that $\kappa (x)>0$ for $x>0$
and $\dint\limits_{0^{+}}\frac{dx}{\kappa (x)}=\infty $. If $f(t)\leq
\dint\limits_{0}^{t}v(s)\kappa (f(s))ds$ for all $t\in \lbrack 0,T]$, then $%
f(t)=0$ for all $t\in \lbrack 0,T]$. \newline
\end{lemma}

Now, for $(t,x)\in \lbrack 0,T)\times 
\mathbb{R}
$, we consider the following form of the driver in (\ref{p0}): 
\begin{equation}
F(t,x,y,z)=f(t,x)+h(t,x)z-\lambda (t,y)-\frac{1}{2}H(t)z^{2},  \label{driver}
\end{equation}%
where the real-valued continuous functions $f,h$ and $\lambda $ on $%
[0,T]\times 
\mathbb{R}
$, and $H:[0,T]\rightarrow 
\mathbb{R}
$ are continuous. The motivation for the choice of such a driver comes from
the stochastic optimal control applications where the control process
appears only in the drift term. Here, the driver function $F$ may neither be
Lipshitz with respect to any of the variables, nor have a linear growth in
any of them. Moreover, we neither impose any convexity/concavity assumption
on $f$ or $\lambda ,$ nor an exponential moment condition on the terminal
condition $g$ or $f(t,x)$\footnote{%
Exponential moment conditions are too strong for many interesting FBSDEs
where the terminal condition depends on an exponential martingale process,
as in the mathematical finance applications}.\ To the\ best of our
knowledge, no existence or uniqueness result is known to cover the BSDEs
with such general drivers even though some special cases were considered in
Cetin (2005), Briand et. al (2007), Briand and Hu (2008) and Richou (2011).

\begin{condition}
(i) $H$ is a positive and continuously differentiable function which is
bounded away from zero. \newline
(ii) $f(t,x)\geq 0$ on $[0,T]\times 
\mathbb{R}
,$\ and satisfies $f(t,X)\in L_{F}^{1}$ where $X$ is as in (\ref{p0}) and .%
\textit{\ }\newline
(iii) \textit{the function }$\lambda $ is such that 
\begin{equation}
2\left\vert u-v\right\vert \cdot \left\vert u\lambda (t,M-\ln
u/H(t))-v\lambda (t,M-\ln v/H(t))\right\vert \leq \varphi (t).\kappa
(\left\vert u-v\right\vert ^{2}),  \label{lambda}
\end{equation}%
for $0\leq t\leq T$ \ and $0<u,v\leq 1$, where the function $\kappa $\
satisfies the conditions given\ in\ Lemma \ref{Bihari} and $\varphi $ is a
continuous function. \newline
(iv) the terminal condition $g(x)$ is bounded from below such that $%
g(X_{T})\in L_{F_{T}}^{1}$. \newline
(v) $h(t,x)$ is bounded on $[0,T]\times 
\mathbb{R}
$ \newline
(v)$^{\prime }$ there is a constant $\gamma \in (0,1)$ such that $%
2H(t)f(t,x)-\frac{h^{2}(t,x)}{\gamma }\geq 0,$\ uniformly on $[0,T]\times 
\mathbb{R}
$\ and $h(t,X)\in L_{F}^{2}.$
\end{condition}

We now state a technical lemma that will be needed in the proof of the main
result of the paper.\newline

\begin{lemma}
\label{kappa} Let $0<r\leq 1$, $0<\epsilon <e^{-r}$ and define a function $%
\kappa ^{\epsilon ,r}(.)$ as 
\begin{equation}
\kappa ^{\epsilon ,r}(x)=\left\{ 
\begin{array}{cc}
x(\ln (x^{-1}))^{r}, & 0<x\leq \epsilon \\ 
\kappa ^{\epsilon ,r}(\epsilon )+\dot{\kappa}^{\epsilon ,r}(\epsilon
)(x-\epsilon ), & x>\epsilon%
\end{array}%
\right.  \label{kap}
\end{equation}%
where $\dot{\kappa}^{\epsilon ,r}(\epsilon )=$\ $\lim\limits_{x\rightarrow
\epsilon ^{-}}\dot{\kappa}^{\epsilon ,r}(x)$. Then $\kappa ^{\epsilon ,r}(.)$
is an increasing, non-negative and concave (differentiable) function
satisfying \textit{\newline
(i) }$\lim\limits_{x\rightarrow 0^{+}}\kappa ^{\epsilon ,r}(x)=0$. \textit{%
\newline
(ii) }For all $0<r<1$ and $0<\epsilon <e^{-r},\ \exists \epsilon _{1}\in
(0,e^{-1})$ such that $\kappa ^{\epsilon ,r}(.)<\kappa ^{\epsilon
_{1},1}(.), $ uniformly in $x.$\textit{\newline
(iii) there exists a constant }$C=C(\epsilon ,r)>0$\textit{\ such that }$%
\left\vert x-y\right\vert \left\vert \kappa ^{\epsilon ,r}(x)-\kappa
^{\epsilon ,r}(y)\right\vert \leq C\kappa ^{\epsilon ,r}(\left\vert
x-y\right\vert ^{2})$, for all $x,y$ in $(0,1]$. In particular, $\left\vert
x-y\right\vert \left\vert \kappa ^{\epsilon ,r}(x)-\kappa ^{\epsilon
,r}(y)\right\vert \leq C_{1}\left\vert x-y\right\vert ^{2}\ln (\left\vert
x-y\right\vert ^{-2})$ also holds, with $C_{1}\leq C(\epsilon ,r).$ \textit{%
\newline
(iv) }$\dint\limits_{0^{+}}\frac{1}{\kappa ^{\epsilon ,r}(x)}dx=\infty $, for%
$\ 0<r\leq 1$.
\end{lemma}

\begin{proof}
For simplicity, we write $\kappa =\kappa ^{\epsilon ,r}$. Note that $\kappa
(x)$\ describes a line with a positive slope $\dot{\kappa}(\epsilon )=[\ln
(\epsilon ^{-1})]^{r}\{1-r/\ln (\epsilon ^{-1})\}$ for $x>\epsilon $.\ It is
straightforward to see that $\dot{\kappa}(x)>0$ for $x\leq \epsilon $, $%
\ddot{\kappa}(x)$ $<0$ for $x>0$. Hence $\kappa (.)$\ is a (strictly)
increasing concave function and the result $\lim\limits_{x\rightarrow
0^{+}}\kappa ^{\epsilon ,r}(x)=0$ in part (i) is a straightforward
application of L'Hopital's rule. Moreover, for fixed $r$,\ the expression $%
(\ln (x^{-1}))^{p}$ is strictly increasing in $p$ for $0<x\leq \epsilon
<e^{-1}$ and $r\leq p\leq 1$. So, the strict inequality $\kappa ^{\epsilon
,r}(x)<\kappa ^{\epsilon _{1},1}(x)$ in (ii) holds for all $0<x\leq \epsilon
_{1}=\epsilon <e^{-1}$. To ensure this inequality is also valid for larger $%
x $ and $\epsilon $\ values, let $e^{-1}\leq \epsilon <e^{-r}.$ Since $\dot{%
\kappa}^{\epsilon ,r}(.)$ is a decreasing function, we have $1-r=\dot{\kappa}%
^{\epsilon ,r}(e^{-1})\geq \dot{\kappa}^{\epsilon ,r}(\epsilon )$. Then we
can select $\epsilon _{1}$ such that $\dot{\kappa}^{\epsilon
_{1},1}(\epsilon _{1})\geq 1-r$. For example, $0<\epsilon _{1}\leq
e^{r-2}<e^{-1}$ will do it, proving the part (ii).\ 

To show (iii), without loss of generality, assume that $0<x<y\leq 1$ and let 
$d=y-x>0$. For $x\geq \epsilon $, we have $\kappa (y)-\kappa (x)=d\dot{\kappa%
}(\epsilon ),$ where $\dot{\kappa}(\epsilon )\leq \lbrack \ln (\epsilon
^{-1})]^{r}\leq C[\ln (d^{-2})]^{r}$, with $C=\max \{1,(\frac{\ln (\epsilon )%
}{2\ln (1-\epsilon )})^{r}\}$, depending on whether $d^{2}\geq \epsilon $
holds \footnote{%
Since the expression $(\frac{\ln (\epsilon )}{2\ln (1-\epsilon )})^{r}$ is
increasing in $r$ and decreasing in $\epsilon $, by choosing $\epsilon
\approx e^{-r}$ for each $r,\ C$ can be selected to be $\lim\limits_{%
\epsilon \rightarrow e^{-1}}\frac{\ln (\epsilon )}{2\ln (1-\epsilon )}%
=\allowbreak 1.\,\allowbreak 090\,1$.}. So, 
\begin{equation*}
d\left\vert \kappa (y)-\kappa (x)\right\vert \leq Cd^{2}[\ln
(d^{-2})]^{r}=C\kappa ^{\epsilon ,r}(d^{2}).
\end{equation*}%
For $x<\epsilon ,$\ since $d<y$, there are two other possible cases, namely, 
$d^{2}\leq x\leq y$ or $x<d^{2}\leq y.$ When $d^{2}\leq x\leq y$, by mean
value theorem, $\kappa (y)-\kappa (x)=d\dot{\kappa}(z),$ for some $z$
between $x$ and $\min \{y,$ $\epsilon \}$. But since $\dot{\kappa}(.)$ is
strictly decreasing on $(0,\epsilon ]$ and $d^{2}\leq x\leq z$, we obtain $%
\kappa (y)-\kappa (x)\leq d\dot{\kappa}(d^{2})\leq d[\ln (d^{-2})]^{r}$. For
the case $x<d^{2}\leq y$, by adding and subtracting $x[\ln (y^{-1})]^{r}$ to 
$\kappa (y)-\kappa (x),$ we get $0<\kappa (y)-\kappa (x)=d[\ln
(y^{-1})]^{r}+x\{[\ln (y^{-1})]^{r}-[\ln (x^{-1})]^{r}\}$, where $\ln
(y^{-1})\leq \ln (d^{-2})<\ln (x^{-1})$. Then the inequality 
\begin{equation*}
\kappa (y)-\kappa (x)<d[\ln (y^{-1})]^{r}\leq d[\ln (d^{-2})]^{r}
\end{equation*}%
easily follows, and hence, when $x<\epsilon $, (iii) holds with $C=1$, .
Moreover, by part (ii), $\exists \epsilon _{1}\in (0,e^{-1})$ such that $%
\kappa ^{\epsilon ,r}(\left\vert x-y\right\vert ^{2})<\kappa ^{\epsilon
_{1},1}(\left\vert x-y\right\vert ^{2})$ and hence the result follows for
all $r\in (0,1]$. The part (iv) is simply a result of part (ii): $\exists
\epsilon _{1}\in (0,e^{-1})$ such that, for all $0<r<1,\delta >0$ and $%
0<\epsilon <e^{-r},\ $%
\begin{equation*}
\dint\limits_{0}^{\delta }\frac{dx}{\kappa ^{\epsilon ,r}(x)}\geq
\dint\limits_{0}^{\delta }\frac{dx}{\kappa ^{\epsilon _{1},1}(x)}\geq
\dint\limits_{0}^{\min \{\delta ,\epsilon _{1})}\frac{-dx}{x\ln (x)}=\infty .
\end{equation*}%
\ 
\end{proof}

\begin{theorem}
\label{Main} For $T>0$ and $p\geq 1,$ let the SDE in (\ref{p0})\textit{\
have a unique solution }$X$\textit{\ in }$L_{F}^{p}[0,T]$\textit{\ with a.s.
continuous paths.\ }Moreover, let the assumptions (i)-(iv), and (v) or (v)$%
^{\prime }$ of Condition 1 hold for the BSDE (\ref{18}) with driver $%
F(t,x,y,z)$\ as in (\ref{driver}). Then the BSDE \textit{(\ref{18})} has at
most one solution $(Y,Z)$\ in $S_{F_{T}}^{1}\times L_{F}^{2}$ such that $Y$
is bounded from below.
\end{theorem}

\begin{proof}
If the pair $(Y,Z)$ is such a solution, let $M$ be a lower bound for $Y$ and
consider the exponential transformation $U(t)\triangleq \exp
(-H(t).(Y(t)-M)) $, for $t\in \lbrack 0,T]$. Clearly, $U(.)$ is bounded a.s.
(between 0 and 1) and $Y(t)=M-\ln U(t)/H(t)$ can be uniquely recovered from $%
U(t)$. The same idea applies to any solution $(Y^{\prime },Z^{\prime })$ to
the equation \textit{(\ref{18}),} and hence the problem reduces to showing
the uniqueness of the solutions to the BSDE for the transformed process $%
U(.) $.\ For simplicity of the notation, we take $M=0$. By Ito's rule and (%
\ref{driver}), a pair $(U,\Lambda )$ with $\Lambda (t)\triangleq -H(t)UZ(t)$
and $U(t)\triangleq \exp (-H(t).Y(t))$\ satisfies the nonlinear BSDE 
\begin{equation}
dU(t)=\{\frac{\dot{H}}{H}\ln U-H\lambda (t,\frac{-\ln U}{H})+Hf(t,X)-\frac{%
h(t,X)\Lambda }{U}\}U(t)dt+\Lambda (t)dW(t)  \label{bsde1}
\end{equation}%
with the terminal condition $U(T)=\exp (-g(X(T)))$ and $0<U(.)\leq 1$ a.s.
on $[0,T]$.

Now, let $(U_{1},\Lambda _{1})$ and $(U_{2},\Lambda _{2})$\ be two (bounded)
solutions to the BSDE \textit{(\ref{bsde1})}. Then, by applying the Ito's
rule to $(U_{1}-U_{2})^{2}$ and rearranging the terms, $P$-a.s, the
expression 
\begin{equation}
\left\vert U_{1}(t)-U_{2}(t)\right\vert ^{2}+\int\limits_{t}^{T}(\Lambda
_{1}-\Lambda
_{2})^{2}(s)ds+\int\limits_{t}^{T}2H(s)f(s,X)(U_{1}-U_{2})^{2}(s)ds
\label{pos}
\end{equation}%
\ can be written as 
\begin{eqnarray}
&&-\int\limits_{t}^{T}2(U_{1}-U_{2})(\Lambda _{1}-\Lambda
_{2})(s)dW(s)-\int\limits_{t}^{T}[2\frac{\dot{H}}{H}(U_{1}-U_{2})(U_{1}\ln
U_{1}-U_{2}\ln U_{2})(s)]ds  \label{ineq0} \\
&&+\int\limits_{t}^{T}2H(U_{1}-U_{2})[U_{1}\lambda (s,\frac{-\ln U_{1}}{H}%
)-U_{2}\lambda (s,\frac{-\ln U_{2}}{H})](s)ds+\int\limits_{t}^{T}2h(s,X)(%
\Lambda _{1}-\Lambda _{2})(U_{1}-U_{2})(s)ds,  \notag
\end{eqnarray}%
a.s. for $0\leq t<T$. Note that the integral $\int%
\limits_{t}^{T}2H(s)f(s,X(s))(U_{1}-U_{2})(s)]^{2}ds$ in \textit{(\ref{pos})}
is non-negative a.s. by the positivity assumptions on $H$ and $f$, implying
that both \textit{(\ref{pos})} and \textit{(\ref{ineq0})} are non-negative.
In \textit{(\ref{ineq0})}, the first integral is a martingale, and by Lemma %
\ref{kappa} with $r=1$ and $\epsilon $ being sufficiently close to $e^{-r}$,
the expression $(U_{1}-U_{2})(U_{1}\ln U_{1}-U_{2}\ln U_{2})$\ in the second
integral satisfies 
\begin{equation*}
\left\vert U_{1}-U_{2}\right\vert \left\vert U_{1}\ln U_{1}-U_{2}\ln
U_{2}(.)\right\vert \leq C\kappa ^{\epsilon ,1}(\left\vert
U_{1}-U_{2}\right\vert ^{2}(.)).
\end{equation*}
Moreover, thanks to the assumption \textit{(\ref{lambda})} for $\lambda ,$
the third integral of \textit{(\ref{ineq0})} is bounded\ by $%
\int\limits_{t}^{T}H(s)\left\vert \varphi (s)\right\vert \kappa (\left\vert
U_{1}(s)-U_{2}(s)\right\vert ^{2})ds$, for some function $\kappa $\ as\ in\
Lemma \ref{Bihari}. Finally, let the assumption (v) of Condition 1 hold and $%
K$ be an upper bound for $\left\vert h(t,x)\right\vert $. Then, applying the
inequality $2\left\vert ab\right\vert \leq \gamma a^{2}+b^{2}/\gamma $ to
the integrand of the last term of \textit{(\ref{ineq0})} with $\gamma =2K$,
we get 
\begin{eqnarray*}
\left\vert \int\limits_{t}^{T}2h(s,X)(\Lambda _{1}-\Lambda
_{2})(U_{1}-U_{2})(s)ds\right\vert  &\leq &K\int\limits_{t}^{T}2\left\vert
(\Lambda _{1}-\Lambda _{2})(U_{1}-U_{2})\right\vert (s)ds \\
&\leq &\frac{1}{2}\int\limits_{t}^{T}\left\vert \Lambda _{1}-\Lambda
_{2}\right\vert ^{2}ds+2K^{2}\int\limits_{t}^{T}\left\vert
U_{1}-U_{2}\right\vert ^{2}ds.
\end{eqnarray*}

Therefore, taking the expected value of both \textit{(\ref{pos}) and } 
\textit{(\ref{ineq0})}, and combining with the terms above, the following
upper bound for $E[\left\vert U_{1}-U_{2}\right\vert
^{2}(t)+\int\limits_{t}^{T}\frac{\left\vert \Lambda _{1}-\Lambda
_{2}\right\vert ^{2}(s)}{2}ds]$ is obtained:%
\begin{equation}
2CE\int\limits_{t}^{T}\left\vert \frac{\dot{H}}{H}\right\vert \kappa
^{\epsilon ,1}(\left\vert U_{1}-U_{2}\right\vert
^{2})(s)ds+E\int\limits_{t}^{T}\left\vert \varphi \right\vert H\kappa
(\left\vert U_{1}-U_{2}\right\vert
^{2})(s)ds+2K^{2}E\int\limits_{t}^{T}\left\vert U_{1}-U_{2}\right\vert
^{2}(s)ds  \notag
\end{equation}%
which is further bounded by $E\int\limits_{t}^{T}\upsilon \xi (\left\vert
U_{1}-U_{2}\right\vert ^{2})(s)ds+2K^{2}\int\limits_{t}^{T}E\left\vert
U_{1}-U_{2}\right\vert ^{2}(s)ds$\ where $\xi (x)=\kappa ^{\epsilon
,1}(x)+\kappa (x)$ is concave and $\upsilon (t)=\max \{H\varphi
(t),2C\left\vert \frac{\dot{H}}{H}(t)\right\vert ,$ satisfying the
assumptions of the Lemma \ref{Bihari}. Now, these bounds imply, in
particular, that 
\begin{equation*}
E\left\vert U_{1}(t)-U_{2}(t)\right\vert ^{2}\leq
E[\int\limits_{t}^{T}\upsilon (s)\xi (\left\vert
U_{1}(s)-U_{2}(s)\right\vert ^{2})ds+2K^{2}\int\limits_{t}^{T}E\left\vert
U_{1}(s)-U_{2}(s)\right\vert ^{2},
\end{equation*}%
and hence by an appeal to the Gronwall's and Jensen's inequalities, we
deduce 
\begin{equation*}
E\left\vert U_{1}-U_{2}\right\vert ^{2}(t)\leq
e^{2K^{2}(T-t)}E[\int\limits_{t}^{T}\upsilon \xi (\left\vert
U_{1}-U_{2}\right\vert ^{2})(s)ds\leq \int\limits_{t}^{T}\upsilon \xi
(E\left\vert U_{1}-U_{2}\right\vert ^{2})(s)ds.
\end{equation*}%
Then, by Bihari's inequality, for all $t$, $E\left\vert
U_{1}(t)-U_{2}(t)\right\vert ^{2}=0$ a.s., implying also that $U_{1}=U_{2}$
a.s. and consequently $\Lambda _{1}=\Lambda _{2}$ a.s.. By transforming back
to $(Y,Z)$, the result follows. The proof is similar when the assumption (v)
of Condition 1 is replaced with the alternate condition (v)$^{\prime }$. In
that case, for $0<\gamma <1,$\ we again apply the inequality $2\left\vert
ab\right\vert \leq \gamma a^{2}+b^{2}/\gamma $ to $2(\Lambda _{1}-\Lambda
_{2})h(U_{1}-U_{2})$\ but instead with the parameters $a=(\Lambda
_{1}-\Lambda _{2})$\ and $b=h(U_{1}-U_{2});$ combine the resulting integrals
with the terms of \textit{(\ref{pos}) }and finally apply the Bihari's and
Jensen's inequalities (without an appeal to the Gronwall's inequality) to
get the result\textit{.}
\end{proof}

\begin{remark}
(a) Some examples for the function $\lambda $, satisfying the condition (iv)
of the Theorem, are given below: \textit{\newline
(i) Let }$\lambda _{1}(t,u)=\alpha (t)u^{r},$ where $\alpha (.)$\ is a
(positive) continuous function and $0<r\leq .$ The corresponding concave
function $\kappa =\kappa _{1}$ in \textit{(\ref{lambda}) }is actually given
by \textit{(\ref{kap}) of Lemma \ref{kappa}: }$\kappa _{1}(x)=$ $\kappa
^{\epsilon ,r}(x)$ for some $0<\epsilon <e^{-r}$\textit{. Note that }$%
\lambda _{1}(t,u)$ is also concave in $u$.\textit{\newline
(ii)Let }$\lambda _{2}(t,u)=e^{-\beta (t)u},$ where $\beta :[0,1]\rightarrow
\lbrack 0,\infty )\ $is a continuous function. Here, $\lambda _{2}(t,.)$ is
a convex function and the corresponding concave function in \textit{(\ref%
{lambda}) is}$\ \kappa _{2}(x)=x.$ \textit{\newline
(iii) Consider }$\lambda _{3}(t,u)=\lambda _{1}(t,u)+\lambda _{2}(t,u),$%
\textit{\ as a sum of a concave and a convex function. Now, the
corresponding }$\kappa _{3}(.)$\textit{\ would be taken as} $\kappa
_{1}(.)+\kappa _{2}(.)$ or $\max (\kappa _{1}(.),\kappa _{2}(.))$.\textit{%
\newline
(iv) Yet another example where the function }$\lambda $\textit{\ is
super-linear in }$\mathit{u}$\textit{\ is }$\lambda _{4}(t,u)=Cu\ln (u^{-1})$
and $\kappa _{4}(.)=Cu\ln (u^{-1})\ln (\ln (u^{-1})$. The reader is
encouraged to find other interesting examples.\textit{\newline
(b) }The existence of a (global) solution under the assumptions of the
Theorem \textit{\ref{Main}}\ (even with a bounded terminal condition and
time-homogenous parameters) is not guaranteed in general. Briand et. al
(2007) provides an example where an exponential moment condition on the
driver is violated. Similarly, the generalizations of the
existence-uniqueness results for the BSDEs with linear growth in $z$ (see
e.g. Fan and Jiang, 2010 and the references there) are not directly
applicable to the transformed BSDE \textit{(\ref{bsde1}) due to the
conditions on the functions }$f(t,x)$\textit{\ and }$\lambda (t,u)$. \textit{%
\newline
(c) }One can perhaps try a combination of the standard localization methods
and the Picard iterations (which also appeared in some of the papers cited
earlier) directly to the original BSDE \textit{(\ref{18}) or to (\ref{bsde1}%
) for the existence part. However, it is not the direction we follow in this
work.} Instead, we will exploit their connections with PDEs of the form (\ref%
{p2})-(\ref{driver}) in the next section by also providing an application to
a stochastic optimal control problem.
\end{remark}

\section{3. The PDE and FBSDE Representations}

In this section, our aim is to show the connections between the solution $%
(Y,Z)$\ of the Markovian FBSDE system (\ref{p0}) and that of the quasilinear
PDEs of the form (\ref{p2})-(\ref{driver}). Note that we haven't assumed any
conditions on the drift and diffusion parameters of the forward process $X$
so far (hence the PDE may be degenerate). Moreover, the conditions that we
imposed on the driver and the terminal condition are more general than the
standard regularity and growth conditions (e.g. Lipshitz condition,
boundedness of the derivatives of the coefficients, linear growth etc.) for
nonlinear PDEs to ensure the existence of a smooth solution to the PDE (\ref%
{p2})-(\ref{driver}). So we may only expect to have a generalized solution
(e.g. a viscosity solution) to such a PDE.

\subsection{3.1 PDE Characterization of the Problem}

By a heuristic application of the seminal result of Pardoux and Peng (1992)
and the setup above, if a function $V(t,x)$\ is a smooth solution to the
equation (\ref{p2}), then the pair $(Y_{t}^{s,x},Z_{t}^{s,x})$ with $%
Y_{t}=V(t,X_{t})$ and $Z_{t}=\sigma (t,X_{t})V_{x}(t,X_{t})$ can be shown to
be a solution to the BSDE 
\begin{eqnarray}
dY_{t}^{s,x} &=&-F(t,X_{t},Y_{t},Z_{t})dt+Z_{t}dW_{t}  \label{18} \\
Y_{T}^{s,x} &=&g(X_{T})  \notag
\end{eqnarray}%
with 
\begin{equation}
X_{t}=X_{t}^{s,x}=x+\int\limits_{s}^{t}\mu
(t,X_{r}^{s,x})dr+\int\limits_{s}^{t}\sigma (r,X_{r}^{s,x})dW_{r},
\label{18b}
\end{equation}%
and$\ F(t,x,y,z)$ as in (\ref{driver}).

\begin{remark}
\label{rep}(a) \textit{The representation of (\ref{18})-(\ref{18b}) as a
FBSDE system is not unique. Another representation may be given by the
following system, by eliminating the drift term of the forward process:} 
\begin{eqnarray}
\hat{X}_{t}^{s,x} &=&x+\int\limits_{s}^{t}\sigma (r,\hat{X}_{r})dW_{r}
\label{19} \\
Y_{t}^{s,x} &=&g(\hat{X}_{T})+\int\limits_{t}^{T}\hat{F}(r,\hat{X}%
_{r},Z_{r})dr-\int\limits_{t}^{T}Z_{r}dW_{r}  \notag
\end{eqnarray}%
\textit{where the new driver function} \textit{is }$\hat{F}%
(t,x,y,z)=F(t,x,y,z)+\frac{\mu (t,x)}{\sigma (t,x)}z$ (as long as the
Girsanov's theorem applies). \textit{Each representation has some advantages
depending on the complexity level of the forward and backward equations in (%
\ref{18})-(\ref{19}). In this section, the representation (\ref{18}) will be
used frequently based on the assumption that the forward state dynamics (\ref%
{18b}) has a unique solution.}\newline
(b) \textit{The existence-uniqueness of the solutions to a particular form
of (\ref{18}) was shown in Cetin (2005, section 2.2), thanks to its
stochastic control interpretation as a solution to the standard LQR problems.%
}
\end{remark}

\begin{corollary}
\label{uniq} Consider the assumptions of the Theorem \textit{\ref{Main} and
let }$F(t,x,y,z)$\textit{\ be given by }(\ref{driver}).\textit{\ }If $%
V(t,x)\in C^{1,2}([0,T]\times 
\mathbb{R}
)$\ satisfies the PDE (\ref{p2}), then we have $V(t,x)=Y_{t}^{t,x}\triangleq
Y^{t,x}(t)$\ for all $(t,x)\in \lbrack 0,T)\times 
\mathbb{R}
,$\ where the pair $(Y_{t}^{s,x},Z_{t}^{s,x})$ given by $Y(t)=V(t,X_{t})$
and $Z(t)=\sigma (t,X_{t})V_{x}(t,X_{t})$\ solves the system \textit{(\ref%
{18})-(\ref{18b}) uniquely}. Moreover, $V$ is the unique solution of the PDE.
\end{corollary}

\begin{proof}
If $V(t,x)$\ is a classical solution to the PDE (\ref{p2}), then define $(%
\bar{Y}_{t}^{s,x},\bar{Z}_{t}^{s,x})$ depending on $X_{t}^{s,x}$
deterministically as $\bar{Y}_{t}=V(t,X_{t})$ and $\bar{Z}_{t}=\sigma
(t,X_{t})V_{x}(t,X_{t}).$ Applying Ito's rule to $\bar{Y}_{t}\triangleq
V(t,X_{t}),$ and by (\ref{p1})\ and (\ref{p2}), we get 
\begin{eqnarray*}
d\bar{Y} &=&\mathsf{L}V(t,X)dt+\sigma (t,X)V_{x}(t,X)dW \\
&=&-F(t,X,V(t,X),\sigma (t,X)V_{x}(t,X))dt+\sigma (t,X)V_{x}(t,X)dW \\
&=&-F(t,X,\bar{Y},\bar{Z})dt+\bar{Z}dW.
\end{eqnarray*}%
So, by the uniqueness of the solutions to \textit{(\ref{18}) }from Theorem %
\ref{Main}\textit{, }the result easily follows.
\end{proof}

\begin{remark}
The converse of the Corollary \ref{uniq} is also true in the sense that if
the triple $(X_{t}^{s,x},Y_{t}^{s,x},Z_{t}^{s,x})$\ solves the system 
\textit{(\ref{18})-(\ref{18b}) and possess some stability and path
regularity properties, then the deterministic function }$V(t,x)$ defined as $%
V(t,x)=Y_{t}^{t,x}$\ is a viscosity solution of the PDE (\ref{p2}). Such a
result is given by Briand and Hu (2008). The uniqueness may require some
extra monotonicity conditions on $F(.,.,y,.)$. We stay working with the
smooth solutions in this work.
\end{remark}

\subsection{3.2. A Stochastic Control Application}

Now consider the following controlled state dynamics $X_{t}=X_{t}^{u}$ with
a control-dependent drift term:%
\begin{eqnarray}
dX_{t} &=&(\mu (t,X_{t})+B(t,X_{t})u_{t})dt+\sigma (t,X_{t})dW(t),
\label{state1} \\
X_{0} &=&x_{0}>0  \notag
\end{eqnarray}%
where $\mu ,\sigma ,B:[0,T]\times 
\mathbb{R}
\rightarrow 
\mathbb{R}
$ are contiuous\ and $u$ belongs to the control space $\mathcal{U}$ of
square integrable real-valued adapted processes such that the equation (\ref%
{state1})\ also has a strong solution $X^{u}\in L_{F}^{2}$. Let the cost
functional be given by 
\begin{equation}
J^{u}(s,x)=E_{s,x}\int\limits_{s}^{T}[(X_{t}-\xi
(t))^{2}+k_{1}(t)u_{t}^{2})]dt+k_{2}(X_{T}-\xi (T))^{2}  \label{cost1}
\end{equation}%
where $k_{1}(.)>0$, $k_{2}\geq 0$, and $\xi (t)$ is a continuous function,
describing the target for the state process $X_{t}=X_{t}^{u}$ to approach or
stay close.\ Define the value function as $V(s,x)=\inf_{u}J^{u}(s,x)$ which
is finite since both $k_{1}(t)u_{t}^{2}$\ and $k_{2}(X_{T}-\xi (T))^{2}$\
are bounded from below. This formulation resembles the stochastic LQR
problems except that here the functions $\mu $ and $\sigma $ need not be
linear in $x$, and $B(t,x)$ may also depend on $x$. Assuming that the SDE (%
\ref{state1}) has a solution for a sufficiently rich set of the control
processes in $\mathcal{U}$, and the optimization problem (\ref{cost1}) is
solvable, we can identify a corresponding FBSDE system to characterize the
solution and solve it numerically.

By a formal application of the dynamic programming principle (DPP) of the
standard stochastic control theory (as in Fleming and Soner, 2006), the
value function should satisfy the HJB equation 
\begin{eqnarray}
v_{t}(t,x)+\frac{1}{2}\sigma ^{2}v_{xx}(t,x)+\inf_{u}\{(x-\xi
(t))^{2}+k_{1}(t)u^{2}+v_{x}(t,x)(\mu (t,x)+B(t,x)u)\} &=&0  \label{hjb0} \\
k_{2}(x-\xi (T))^{2} &=&v(T,x)  \notag
\end{eqnarray}%
where the infimum of the (Hamiltonian) expression$\ (x-\xi
(t))^{2}+k_{1}(t)u^{2}+v_{x}(t,x)(\mu (t,x)+B(t,x)u)$\ is obtained with $%
u^{\ast }(t,x)=-\frac{Bv_{x}(t,x)}{2k_{1}(t)}$. By writing this candidate
optimal control in the equation (\ref{hjb0}), we obtain a quasilinear PDE of
the form (\ref{p2}), given by (\ref{pde1}) below, with $F(t,x,y,z)=(x-\xi
(t))^{2}-\frac{1}{2}H(t,x)z^{2}$, where $H(t,x)=\frac{1}{2k_{1}(t)}(\frac{%
B(t,x)}{\sigma (t,x)})^{2}$. Note that the function $F$ is independent of $y$%
\ \footnote{%
It would depend on y linearly, if we considered a time-discounted cost
function.}\ and $H(t,x)$ may depend on $x$. In general, a classical solution
to the equation (\ref{hjb0})\ is not guaranteed to exist.\ However if $%
H(t,x)=H(t)$, and if a smooth solution to the corresponding HJB PDE (\ref%
{pde1}) exists, then\ the results of the previous section apply and we have
the following result:

\begin{theorem}
In the setting above, suppose that $H(t,x)$ is time-dependent only: $%
H(t,x)=H(t)$ and $F(t,x,y,z)=(x-\xi (t))^{2}-\frac{1}{2}H(t)z^{2}$.\ Assume
that for all $p\geq 2$, the SDE (\ref{state1})\ has a unique square
integrable solution $X^{u}\in S_{F}^{p}$,\ for $u=0$ and $u=u^{\ast }=-\frac{%
Bv_{x}(t,X_{t})}{2k_{1}(t)}$ where $v(t,x)\in C^{1,2}[0,T]\times 
\mathbb{R}
$ satisfies the quasilinear PDE 
\begin{equation}
v_{t}(t,x)+\frac{1}{2}\sigma ^{2}v_{xx}(t,x)+F(t,x,v,\sigma v_{x})=0,\text{ }%
v(T,x)=k_{2}(x-\xi (T))^{2}\text{.}  \label{pde1}
\end{equation}%
Moreover, let $\tilde{X}_{t}$\ denote the solution to the SDE (\ref{state1})
for $u=0$. Then,\textit{\newline
}(i) The pair $(Y_{t}^{s,x},Z_{t}^{s,x})$ with $Y_{t}=v(t,\tilde{X}_{t})$
and $Z_{t}=\sigma (t)v_{x}(t,\tilde{X}_{t})$ is a (unique) continuous
solution to the BSDE 
\begin{equation}
dY_{t}^{s,x}=-F(t,\tilde{X}_{t},Z_{t})dt+Z_{t}dW_{t}\text{, }%
Y_{T}^{s,x}=k_{2}(X(T)-\xi (T))^{2}  \label{bsde2}
\end{equation}%
in in $S_{F_{T}}^{1}\times L_{F}^{2}$ such that $Y$ is bounded from below.%
\textit{\newline
}(ii) The value function is given by $v(t,x)$ which is the unique smooth
solution of the PDE (\ref{pde1}) and satisfies $v(t,x)=Y_{t}^{t,x},$ for $%
x\in 
\mathbb{R}
$ and $t\in \lbrack 0,T)$.
\end{theorem}

\begin{proof}
The part (i) directly follows from Theorem \ref{Main}, representations (\ref%
{18})-(\ref{18b}) and Corollary \ref{uniq}. When $u^{\ast }$\ is an
admissible control and value function is well-defined (finite), part (ii) is
a result of Corollary \ref{uniq} and the arguments of the stochastic control
theory for the classical solutions of the HJB equations.
\end{proof}

\begin{remark}
(a) Ideally, an applicable "verification" theorem for the control problem or
some a priory bounds for the processes $X,Y$ and $Z$ would be needed (since
we haven't assumed any Lipshitz or growth conditions on the SDE (\ref{state1}%
) explicitly) to get part (ii) of Theorem. In most cases, the value function
will be a viscosity solution to the PDE by a "formal" appeal to a version of
the DPP, if available. \textit{\newline
(b) Under the Lipshitz conditions on }$\mu $\textit{\ and }$\sigma $\textit{%
\ }and a boundedness assumption on $\sigma ,$\textit{\ Fuhrman et. al (2006)
showed that the value function is given by the maximal solution of the BSDE (%
}\ref{bsde2}), using some localization arguments. They also provide the LQR
example as a special case and consider more general applications where the
control set is constrained to take values from a closed set of $%
\mathbb{R}
$. The uniqueness to the solutions of the BSDEs (and hence the corresponding
PDEs) related to the LQR problems was also reported in Cetin (2005), by
exploiting the regularity properties of the explicit solution for the value
function.
\end{remark}

\begin{example}
Consider the following perturbed version of the LQR problem: 
\begin{eqnarray}
dX_{t} &=&(A(t)X_{t}-\delta X_{t}^{3}+B(t)u_{t})dt+\sigma (t)dW(t),
\label{pert} \\
X_{0} &=&x_{0}>0  \notag
\end{eqnarray}%
where the time dependent functions $A,B$ and $\sigma $\ are continuous,$\ B$
and $\sigma $\ are bounded away from zero on the interval $[0,T]$, and $u$
belongs to the control space $\mathcal{U}$ as before. The term $\delta $\ is
a small perturbation constant, so the system reduces to a linear one with
Lipshitz coefficients when $\delta =0$. Even though the standard
(unperturbed) LQR problems have an explicit quadratic form as a solution,
this perturbed version of the problem cannot be solved explicitly. Using the
same arguments above, the corresponding HJB equation is given by 
\begin{equation}
v_{t}(t,x)+\frac{1}{2}\sigma ^{2}(t)v_{xx}(t,x)+\inf_{u}\{(x-\xi
(t))^{2}+k_{1}(t)u^{2}+v_{x}(t,x)(A(t)x-\delta x^{3}+B(t)u)=0  \label{hjb}
\end{equation}%
with $v(T,x)=k_{2}(x-\xi (T))^{2}$. When the terminal condition is bounded
(e.g. when $k_{2}=0$, as in Tsai, 1978), using the methods of the parabolic
PDEs, it can be shown to have a smooth solution. For more general functions,
even when the PDE is uniformly parabolic, the existence of a classical
solution is not guaranteed in general. To prove that the equation (\ref{pert}%
)\ also has a square integrable solution $X^{u^{\ast }}$ corresponding to
the (feedback) control $u^{\ast }(t)=-\frac{B}{2k}(t)v_{x}(t,X(t))$, we may
need some a priory estimates on the (potentially viscosity) solutions of (%
\ref{Pde}). However, if the solution is smooth, the uniqueness follows from
Corollary \ref{uniq}.
\end{example}

\begin{theorem}
Consider the perturbed state dynamics (\ref{pert})\ together with the cost
function (\ref{cost1}) and the value function $V^{\delta }(s,x)$. Then 
\textit{\newline
}(i) The value function $V(s,x)$ is the unique smooth solution to the HJB PDE%
\begin{eqnarray}
v_{t}+\frac{1}{2}\sigma ^{2}(t)v_{xx}+(x-\xi (t))^{2}-(A(t)x-\delta
x^{3})v_{x}-\frac{1}{2}C(t)v_{x}^{2} &=&0  \label{Pde} \\
v(T,x) &=&0  \notag
\end{eqnarray}%
\ where $C(t)=(B^{2}/2k_{1})(t)$ over $[0,T]$. \textit{\newline
}(ii) Let $H(t)=\frac{C(t)}{\sigma ^{2}(t)}$ satisfy the assumption of
Condition 1 (i). Then the triple $(\tilde{X}%
_{t}^{s,x},Y_{t}^{s,x},Z_{t}^{s,x})$ with $Y_{t}=V(t,\tilde{X}_{t})$ and $%
Z_{t}=\sigma (t)V_{x}(t,\tilde{X}_{t})$ is a (unique) solution to the FBSDE
system 
\begin{eqnarray}
dY_{t}^{s,x} &=&-F(t,\tilde{X}_{t},Z_{t})dt+Z_{t}dW_{t}\text{, }Y_{T}^{s,x}=0
\label{fbsde2} \\
\tilde{X}_{r} &=&x-\int\limits_{s}^{t}(A\tilde{X}_{r}-\delta \tilde{X}%
_{r}^{3})dr+\int\limits_{s}^{t}\sigma (r)dW_{r},  \notag
\end{eqnarray}%
where $F(t,x,z)=(x-\rho (t))^{2}-\frac{H(t)z^{2}}{2}$. Moreover, $%
V(t,x)=Y_{t}^{t,x},$ for $(t,x)\in \lbrack 0,T)\times 
\mathbb{R}
$.
\end{theorem}

\begin{proof}
For part (i), note that the derivative of the function $\mu (x)=Ax-\delta
x^{3}$ is bounded above and $x\mu (x)$ can be bounded from above by $\alpha
-\beta x^{2}$, for some positive constants $\alpha $ and $\beta $. So using
the Lyapunov conditions for locally Lipshitz parameters, the forward SDE in (%
\ref{fbsde2}) can be shown to have a unique global solution $\tilde{X}\in
S_{F}^{p}$, for all $p\geq 1$. By the relevant DPP results for bounded
terminal value problems, if the value function $V(t,x)$ is sufficiently
smooth, then it should satisfy the HJB PDE (\ref{Pde}) together with the
candidate optimal control $u^{\ast }$. Since the time dependent parameters
are (uniformly) continuous on $[0,T]$, by following the similar steps as in
Tsai (1978), one can show the existence of a smooth solution $v(t,x)$ to the
PDE (\ref{Pde}), too. Then by Theorem \ref{Main}, representations (\ref{18}%
)-(\ref{18b}) and the Corollary \ref{uniq}, the BSDE in (\ref{fbsde2}) has a
unique solution $(Y,Z)$ such that $v(t,x)=Y_{t}^{t,x}=V(t,x)$ and $\sigma
(t)V(t,X(t))=Z(t)$. Alternatively, it can be inferred from an applicable
verification theorem, e.g. as in Tsai (1978) or Fleming and Soner (2006).
\end{proof}

\begin{remark}
Such FBSDE representations would be very helpful to solve these type of
nonlinear PDEs (and control problems) numerically, especially in higher
dimensional cases. It is especially useful if the nature of the PDE solution
is not known explicitly and can be inferred from the properties of the
numerical solution to the corresponding FBSDE system. We leave the
discussion of the numerical solutions and their stability/convergence
properties to some subsequent work, including Cetin (2012).
\end{remark}

\renewcommand{\baselinestretch}{1}\small\normalsize%

\end{document}